\documentclass[11pt]{amsart}
\usepackage[left=1in,right=1in,top=1in,bottom=1in]{geometry}               
\usepackage{graphicx, enumerate, amsmath, color, tikz}
\usepackage[pdf]{pstricks}
\usepackage{amssymb}
\usepackage{epstopdf}
\newtheorem{theorem}{Theorem}[section]

\newtheorem{lemma}[theorem]{Lemma}
\newtheorem{corollary}[theorem]{Corollary}

\theoremstyle{definition}

\title{Monochromatic cycle partitions in local edge colourings}
\author{David Conlon}
\address{Mathematical Institute\\
University of Oxford\\
Andrew Wiles Building\\
Radcliffe Observatory Quarter\\
Woodstock Road\\
Oxford OX2 6GG\\
United Kingdom}
\email{david.conlon@maths.ox.ac.uk}

\author{Maya Stein}
\address{Centro de Modelamiento Matem\'atico\\
Universidad de Chile\\
Beauchef 851\\
Piso 7\\
Santiago Centro\\
RM\\
Chile}
\email{mstein@dim.uchile.cl}

\thanks{The first author is supported by a Royal Society University
Research Fellowship and the second author is supported by the 
Fondecyt grants 11090141 and 1140766.}

\date{}                                          

\begin{document}
\maketitle

\begin{abstract}
An edge colouring of a graph is said to be an $r$-local colouring if the edges incident to any vertex are coloured with at most $r$ colours. Generalising a result of Bessy and Thomass\'e, we prove that the vertex set of any $2$-locally coloured complete graph may be partitioned into two disjoint monochromatic cycles of different colours. Moreover, for any natural number $r$, we show that the vertex set of any $r$-locally coloured complete graph may be partitioned into $O(r^2 \log r)$ disjoint monochromatic cycles. This generalises a result of Erd\H{o}s, Gy\'arf\'as and Pyber.
\end{abstract}

\section{Introduction}

A well-known result of Erd\H{o}s, Gy\'arf\'as and Pyber~\cite{EGP91} says that there exists a constant $c(r)$, depending only on $r$, such that if the edges of the complete graph $K_n$ have been coloured with $r$ colours, then the vertex set of $K_n$ may be partitioned into at most $c(r)$ disjoint monochromatic cycles, where we allow single vertices and edges to be cycles. Moreover, they conjectured that this result should hold with $c(r) = r$. 

The $r = 2$ case of this conjecture, attributed (see \cite{Ay79}) to Lehel, is slightly more specific, asking that the vertex set  be partitioned into two disjoint cycles of different colours,  where we now consider the empty set to be a cycle. This conjecture was proved, for $n$ sufficiently large, by \L uczak, R\"odl and Szemer\'edi \cite{LRS98} and later by Allen \cite{A08}, though starting from a much smaller value of $n$. For all $n$, the conjecture was finally resolved by Bessy and Thomass\'e \cite{BT10}.

For $r \geq 3$, the conjecture was shown to be false by Pokrovskiy \cite{P14}. However, for the case $r = 3$, there are some partial results showing that the conjecture is still very close to being true~\cite{GRSS11, P14}. In general, the best known upper bound on $c(r)$, proved by Gy\'arf\'as, Ruszink\'o, S\'ark\"ozy and Szemer\'edi~\cite{GRSS06}, is $O(r \log r)$ and, despite Pokrovskiy's counterexample, it seems likely that an approximate version of the original conjecture remains true.

We consider a generalisation of this monochromatic cycle partition question to graphs with locally bounded colourings. We say that an edge colouring of a graph is an {\it $r$-local colouring} if the edges incident to any vertex are coloured with at most $r$ colours. Note that we do not restrict the total number of colours. Somewhat surprisingly, we prove that even for local colourings, a variant of the Erd\H{o}s--Gy\'arf\'as--Pyber result holds. 

\begin{theorem} \label{thm:rcolours}
The vertex set of any $r$-locally coloured complete graph  may be partitioned into $O(r^2 \log r)$ disjoint monochromatic cycles.
\end{theorem}

For $r = 2$, we have the following more precise theorem, which directly generalises the result of Bessy and Thomass\'e.

\begin{theorem} \label{thm:2colours}
The vertex set of any $2$-locally coloured complete graph may be partitioned into two disjoint monochromatic cycles of different colours.
\end{theorem}

We will prove Theorem~\ref{thm:rcolours} in the next section by building on ideas of Erd\H{o}s, Gy\'arf\'as and Pyber and using results on local Ramsey numbers.  In Section~\ref{sec:2colours}, we will prove Theorem \ref{thm:2colours} using the result of Bessy and Thomass\'e as a black box. We conclude with some further remarks. In particular, we will discuss a further extension of Theorem~\ref{thm:2colours} to $2$-mean colourings.

\section{$r$-local colourings} \label{sec:rcolours}

In this section, we will prove Theorem~\ref{thm:rcolours} using a proof strategy similar to that of Erd\H{o}s, Gy\'arf\'as and Pyber but with several elements added which are specific to local colourings. The first ingredient needed by Erd\H{o}s, Gy\'arf\'as and Pyber is a result which says that if we have an unbalanced bipartite graph whose edges have been $r$-coloured then all of the vertices in the smaller set can be covered by a bounded number of cycles in terms of $r$. We will begin by proving an analogue of this result for local colourings.

We will need the following elementary lemma due to P\'osa (see Problem 8.3 in \cite{L07}).

\begin{lemma} \label{lem:Posa}
If a graph has independence number $\alpha$, then the vertex set of the graph may be partitioned into at most $\alpha$ disjoint cycles.
\end{lemma}

It will be useful to introduce some notation. Given a vertex $x$ and a colour $c$, we let $N_c(x)$ be the neighbourhood of $x$ in colour $c$. In the following proof, we will have a collection of colours $c_1, c_2, \dots, c_r$ and we will simply write $N_i(x)$ for $N_{c_i}(x)$.

\begin{lemma} \label{lem:patching}
Suppose that $A$ and $B$ are vertex sets with $|B| \leq |A|/r^{r+3}$ and the edges of the complete bipartite graph between $A$ and $B$ are $r$-locally coloured. Then all vertices of $B$ can be covered with at most $r^2$ disjoint monochromatic cycles.
\end{lemma}

\begin{proof}
Fix a vertex $b_1$ in $B$. Since $b_1$ is incident to at most $r$ colours, there exists a colour $c_1$ such that the neighbourhood $N_1(b_1)$ of $b_1$ in colour $c_1$ has size at least $|A|/r$. We let $A_1 = N_1(b_1)$. 

Suppose now that $b_1, b_2, \dots, b_i$ are distinct vertices in $B$, $c_1, c_2, \dots, c_i$ are distinct colours and that $A_i = \cap_{j=1}^i N_j(b_j)$ satisfies $|A_i| \geq |A|/r^i$. If, for all $b \in B \setminus \{b_1, \dots, b_i\}$, there exists a colour $c(b) \in \{c_1, c_2, \dots, c_i\}$ such that 
\[|A_i \cap N_{c(b)}(b)| \geq \frac{|A_i|}{r},\]
then, for each $j = 1, 2, \dots, i$, we let $B_j = \{b : c(b) = c_j\}$ and $A' = A_i$. Otherwise, there exists a vertex $b_{i+1} \in B\setminus \{b_1, b_2, \dots, b_i\}$ and a colour $c_{i+1} \not\in \{c_1, c_2, \dots, c_i\}$ such that $|A_i \cap N_{i+1}(b_{i+1})| \geq \frac{|A_i|}{r}$. Letting 
\[A_{i+1} = A_i \cap N_{i+1}(b_{i+1}) = \cap_{j=1}^{i+1} N_j(b_j),\]
we see that $|A_{i+1}| \geq |A|/r^{i+1}$.

Suppose now that the process continues until we have defined a set $A_r = \cap_{j=1}^r N_j(b_j)$ with $|A_r| \geq |A|/r^r$. Since every vertex in $A_r$ is adjacent to $r$ different colours and the graph is $r$-locally coloured, we must have that all edges between $A_r$ and $B$ are coloured in $c_1, c_2, \dots, c_r$. If this is the case, then, for each vertex $b \in B$, there exists a colour $c(b) \in \{c_1, c_2, \dots, c_r\}$ such that $|A_r \cap N_{c(b)}(b)| \geq |A_r|/r$. For each $j = 1, 2, \dots, r$, we let $B_j = \{b : c(b) = c_j\}$ and $A' = A_r$.

We may now assume that we have a subset $A'$ of $A$ with $|A'|\geq |A|/r^r$ and a partition of $B$ into pieces $B_1, B_2, \dots, B_r$ (some of which may be empty) such that, for each $j = 1, 2, \dots, r$, every vertex $b \in B_j$ is adjacent to at least $|A'|/r$ vertices of $A'$ in colour $c_j$. We define a graph $G_j$ on vertex set $B_j$ by joining $x, y \in B_j$ if and only if 
\[|A' \cap N_j(x) \cap N_j(y)| \geq |A'|/r^3.\]
We claim that the graph $G_j$ contains no independent set of order $r + 1$. Suppose, for the sake of contradiction, that $x_1, x_2, \dots, x_{r+1} \in B_j$ are the vertices of an independent set of order $r+1$. Then
\begin{align*}
|A'| & \geq \left|\cup_{k=1}^{r+1} \left(A' \cap N_j(x_k)\right)\right| \\
& \geq (r+1) \cdot \frac{|A'|}{r} - \sum_{1 \leq k < \ell \leq r+1} |A' \cap N_j(x_k) \cap N_j(x_{\ell})|\\
& \geq |A'|\left(\frac{r+1}{r} - \frac{\binom{r+1}{2}}{r^3}\right) > |A'|,
\end{align*}
a contradiction. By Lemma \ref{lem:Posa}, it follows that the vertex set $B_j$ can be partitioned into at most $r$ disjoint cycles from $G_j$. Using the definition of $G_j$ and the fact that $|A' \cap N_j(x) \cap N_j(y)| \geq |A'|/r^3 \geq |B|$ for any two adjacent $x,y$ from $B_j$, it is now easy to conclude that the vertices of $B$ may be covered using at most $r^2$ disjoint monochromatic cycles.
\end{proof}

The {\it $r$-colour Ramsey number} of a graph $H$, denoted $R_r(H)$, is the smallest $n$ such that in any $r$-colouring of the edges of $K_n$ there is guaranteed to be a monochromatic copy of $H$. The local analogue of this concept, known as the {\it $r$-local Ramsey number} and denoted $R_{r-loc}(H)$, is the smallest $n$ such that in any $r$-local colouring of the edges of $K_n$ there is guaranteed to be a monochromatic copy of $H$. That the local Ramsey number exists was first proved by Gy\'arf\'as, Lehel, Schelp and Tuza~\cite{GLST87}. We will need the following result of Truszczynski and Tuza~\cite{TT87}, which says that for connected graphs the ratio of $R_{r-loc}(H)$ and $R_r(H)$ is bounded in terms of $r$. 

\begin{lemma} \label{lem:localtoglobal}
For any connected graph $H$,
\[\frac{R_{r-loc}(H)}{R_r(H)} \leq \frac{r^r}{r!} \leq e^r.\]
\end{lemma}

The following lemma generalises a result of Bollob\'as, Kostochka and Schelp \cite{BKS00}. For a class of graphs $\mathcal{H}$, the {\it $r$-local Ramsey number} $R_{r-loc}(\mathcal{H})$ is the smallest $n$ such that in any $r$-local colouring of the edges of $K_n$ there is guaranteed to be a monochromatic copy of some graph $H \in \mathcal{H}$.

\begin{lemma} \label{lem:denselocal}
Suppose that $\mathcal{H}$ is a class of graphs and $c$ and $\epsilon$ are positive constants such that for all $n$ any graph on $n$ vertices with at least $c n^{2 - \epsilon}$ edges contains a graph from $\mathcal{H}$. Then
\[R_{r-loc}(\mathcal{H}) \leq (4 c r)^{1/\epsilon}.\]
\end{lemma}

\begin{proof}
Suppose that the edges of $K_n$ have been $r$-locally coloured with at most $s$ colours, which we may assume are $\{1, 2, \dots, s\}$. For each $i = 1, 2, \dots, s$, let $e_i$ be the number of edges in colour $i$ and $v_i$ the number of vertices which are incident with an edge of colour $i$. If there is no $H \in \mathcal{H}$ in colour $i$, then $e_i < c v_i^{2 - \epsilon}$ for each $i = 1, 2, \dots, s$. Since also $\sum_{i=1}^s v_i \leq r n$, we have
\[\frac{n^2}{4} \leq \binom{n}{2} = \sum_{i=1}^s e_i < c \sum_{i=1}^s v_i^{2 - \epsilon} \leq c (\max_{1 \leq i \leq s} v_i^{1-\epsilon}) \sum_{i=1}^s v_i \leq c r n^{2 - \epsilon}.\]
This implies that $n < (4 c r)^{1/\epsilon}$ and the result follows.
\end{proof}

We will only need the following corollary.

\begin{corollary} \label{cor:longcycles}
Let $\mathcal{C}_\ell$ be the collection of all cycles of length at least $\ell$. Then
\[R_{r-loc}(\mathcal{C}_\ell) \leq 2 \ell r.\]
\end{corollary}

\begin{proof}
A classical result of Erd\H{o}s and Gallai \cite{EG59} shows that if a graph on $n$ vertices contains at least $\ell n/2$ edges, then it contains a cycle of length at least $\ell$. The result then follows from applying Lemma~\ref{lem:denselocal} with $\epsilon = 1$ and $c = \ell/2$.
\end{proof}

For any natural number $k$, we define the triangle cycle $T_k$ to be the graph with vertex set $\{u_1, u_2, \dots, u_k\} \cup \{v_1, v_2, \dots, v_k\}$, where $u_1, u_2, \dots, u_k$ form a cycle of length $k$ and $v_i$ is joined to $u_i$ and $u_{i+1}$ (with addition taken modulo $k$). That is, as the name suggests, we have a cycle formed from triangles. An important property of these graphs is that we can remove any subset of $\{v_1, v_2, \dots, v_k\}$ and still find a cycle through all of the remaining vertices. The final ingredient we will need is a straightforward lemma of Erd\H{o}s, Gy\'arf\'as and Pyber \cite{EGP91} about the Ramsey number of these triangle cycles.

\begin{lemma} \label{lem:tricycle}
\[R_{r}(T_k) = O(k r^{3r}).\]
\end{lemma}

We are now ready to prove Theorem~\ref{thm:rcolours}. 

\vspace{3mm}

\noindent
{\it Proof of Theorem~\ref{thm:rcolours}.}
Let $K_n$ be a complete graph whose edges have been $r$-locally coloured. Combining Lemmas~\ref{lem:localtoglobal} and~\ref{lem:tricycle}, we see that
\[R_{r-loc}(T_k) \leq e^r R_r(T_k) = O(k r^{4r}).\]
Therefore, there is a monochromatic triangle cycle $T_k$ with $k = \Omega(n/r^{4r})$ in our colouring of $K_n$. We let $A$ be the subset of this triangle cycle corresponding to the vertex set $\{v_1, v_2, \dots, v_k\}$, that is, the collection of vertices which are not on the shortest cycle. We now restrict our attention to the complete graph on the vertex set $V(K_n)\setminus V(T_k)$.

By Corollary~\ref{cor:longcycles}, any $r$-locally coloured $K_t$ contains a monochromatic cycle of length at least $t/2r$. Removing the vertices of this cycle leaves an $r$-locally coloured complete graph with at most $t(1 - 1/2r)$ vertices. If we start with the vertex set $V(K_n)\setminus V(T_k)$ and apply this observation $s$ times, we remove $s$ disjoint monochromatic cycles and leave an $r$-locally coloured complete graph with at most
\[n\left(1 - \frac{1}{2r}\right)^s \leq n e^{-s/2r}\]
vertices. Therefore, for $s = c r^2 \log r$ and $c$ sufficiently large, we see that the remaining set of vertices has size at most $|A|/r^{r+3}$, where we used that $|A| = k = \Omega(n/r^{4r})$. We let $B$ be this remaining set.

We now apply Lemma~\ref{lem:patching} to the bipartite graph between $A$ and $B$. This implies that there is a collection of at most $r^2$ disjoint monochromatic cycles which cover all vertices of $B$. Though we have deleted a vertex subset $A'$ of $A$, there is still a monochromatic cycle covering all of the vertices in the set $V(T_k) \setminus A'$. Altogether, we have partitioned the vertex set of $K_n$ using at most
\[cr^2 \log r + r^2 + 1 = O(r^2 \log r)\]
disjoint monochromatic cycles, completing the proof. \qed

\section{$2$-local colourings} \label{sec:2colours}

In this section we will prove Theorem~\ref{thm:2colours}, that the vertex set of any $2$-locally coloured complete graph may be partitioned into two disjoint monochromatic cycles of different colours. We begin with a simple corollary of Bessy and Thomass\'e's result that every $2$-coloured complete graph may be partitioned into two disjoint monochromatic cycles of different colours. Throughout this section, we will say that a colour $\alpha$ \emph{sees} a vertex $v$, and vice versa, if there is an edge of colour $\alpha$ incident with $v$.

\begin{lemma}\label{lem:1coverall}
Let $K_n$ be $2$-locally coloured such that there is a colour $\alpha$ which sees all vertices. Then there are two disjoint monochromatic cycles of different colours, one of these $\alpha$, that together cover all of $V(K_n)$.
\end{lemma}

\begin{proof}
Let $\beta$ be the union of all colours other than $\alpha$. Apply the result of Bessy and Thomass\'e to find two disjoint cycles covering all of $K_n$, one in colour $\alpha$, the other in colour $\beta$. However, the cycle in colour $\beta$ must be monochromatic in the original colouring, since every vertex sees at most one colour different from $\alpha$.
\end{proof}

We also note a slight strengthening of Lemma~\ref{lem:1coverall} which we will need in the next section to prove an extension of Theorem~\ref{thm:2colours} to mean colourings. In  this lemma, the colouring is not required to be a $2$-local colouring, in that we allow a single vertex $v$ to see more than two colours.

\begin{lemma}\label{lem:onemore}
Let the edges of $K_n$ be coloured in such a way that each vertex except possibly $v$ sees at most one colour other than $\alpha$. Then there are two disjoint monochromatic cycles of different colours, one of these $\alpha$, that together cover all of $V(K_n)$.
\end{lemma}

\begin{proof}
Let $\beta$ be the union of all colours other than $\alpha$. Apply the result of Bessy and Thomass\'e to find two disjoint cycles covering all of $K_n$, one in colour $\alpha$, the other in colour $\beta$. If the cycle of colour $\alpha$ contains $v$, we are done as in the previous lemma. If the cycle of colour $\beta$ contains $v$, then we may write the vertices of this cycle as $v_1, v_2, \dots, v_t$, where $v_1 : = v$ and $v_i v_{i+1}$ is an edge for all $i = 1, 2, \dots, t$ (with addition taken modulo $t$). But then, since each vertex $v_i$ with $2 \leq i \leq t$ sees at most one colour other than $\alpha$, we must have that the edges $v_{i-1} v_i$ and $v_i v_{i+1}$ have the same colour for all $2 \leq i \leq t$. This implies that the cycle is monochromatic, completing the proof.
\end{proof}

Suppose now that we have a complete bipartite graph between sets $A$ and $B$ and each of $A$ and $B$ contain paths. The following simple lemma gives a condition under which these paths may be combined into a spanning cycle on $A \cup B$. Note that for a path $P$, we let $|P|$ be the number of vertices on $P$. 

\begin{lemma}\label{lem:pathsBIP}
Let $G$ be a graph and let $A$ and $B$ be disjoint subsets of $V(G)$ such that $G[A,B]$ is complete. Let $P_A$ and $P_B$ be paths in $A$ and  $B$, respectively, with $P_A, P_B\neq\emptyset$. If $|B-P_B|\leq |A-P_A| \leq |B|-1$, then  $A\cup B$ has a spanning cycle.
\end{lemma}

\begin{proof}
We form a path by first following $P_A$ and then alternating between $B$ and $A$ until we cover all of $A$. While doing so, we prefer vertices of $B$ not covered by $P_B$ but if we have to use vertices from $P_B$ we use them in the order they lie on $P_B$. This gives a path $P$ of order $2|A|-|P_A|$, starting and ending in $A$. Since $|A-P_A| \leq |B|-1$, we know $P$ covers all of $A$ and, since $|B-P_B|\leq |A-P_A|$, the path $P$ also covers all of $B-P_B$. Furthermore, $P$ possibly covers some, but not all, of~$P_B$. We may therefore connect $P$ with the remains of $P_B$ to complete a cycle covering all of $A\cup B$.
\end{proof}

In the next lemma we generalize the situation above to three sets and two monochromatic cycles. 

\begin{lemma}\label{lem:pathsTRI}
Let $G$ be a graph whose edges are coloured with colours $1$ and $2$ and let $A_1, A_2$ and $B$ be disjoint subsets of $V(G)$. For $i=1,2$, let  $G[A_i,B]$ be complete in colour $i$, let $P_{A_i}\neq\emptyset$ be a path of colour~$i$ in~$A_i$ and let $P_B^i$  be a path of colour $i$ in $B$. If 
\begin{enumerate}[(a)]
\item $P_B^1$ and $P_B^2$ partition $B$, and\label{parti}
\item $ |A_1-P_{A_1}| + |A_2-P_{A_2}| +2\leq |B|,$\label{size}
\end{enumerate} 
then  $A_1\cup A_2\cup B$ has a partition into two monochromatic cycles, one of each colour.
\end{lemma}

\begin{proof}
First of all, note that we may assume that the paths $P_B^i$ are non-empty. Otherwise, since $|P_B^1|+|P_B^2|=|B|\geq 2$ by~\eqref{parti} and~\eqref{size}, we can slightly modify our paths so that both are non-empty, without losing the conditions of the lemma.

Take a subpath $P_i$ of $P^i_B$ of order $\min\{|P^i_B|, |B|-|A_{3-i}-P_{A_{3-i}}|-1\}$ for  $i=1,2$. Note that $P_i\neq\emptyset$ since we assumed  $P_B^i\neq\emptyset$ and also $|B|-|A_{3-i}-P_{A_{3-i}}|-1\geq 1$ by~\eqref{size}. Observe, by~\eqref{parti} and~\eqref{size}, that we have $P_i=P^i_B$ for at least one of $i=1,2$. Otherwise, we would have $|P^i_B|\geq |B|-|A_{3-i}-P_{A_{3-i}}|$ for $i=1,2$ and thus $|B|=|P^1_B|+|P^2_B|\geq 2|B|-|A_{1}-P_{A_{1}}|-|A_{2}-P_{A_{2}}|\geq |B|+2$, a contradiction. Without loss of generality, we may therefore assume that $P_2 = P^2_B$. 

Clearly, either  $|P_1|= |B|-|A_{2}-P_{A_{2}}|-1$ or $P_1=P^1_B$ (or both). In the first case, we use~\eqref{size} to see that $|P_1|\geq |A_{1}-P_{A_{1}}|+1$. In the second case, we see that $P_2=P^2_B$ implies that $|P^2_B|\leq |B|-|A_{1}-P_{A_{1}}|-1$, by the choice of $P_2$. 
Thus, using~\eqref{parti}, we get $|P_1|=|P^1_B|=|B|-|P^2_B| \geq |A_{1}-P_{A_{1}}|+1$.
In either case, we obtain
\begin{equation}\label{P1}
|P_1|\geq |A_{1}-P_{A_{1}}|+1.
\end{equation}

Let $B_1:=V(P_1)$ and $B_2:=B-B_1$. Note that $ |B_i-P_i|\leq  |A_i-P_{A_i}| \leq |B_i|-1$ for $i=1,2$. The first inequality holds for $i=1$ by the definition of $B_1$ and for $i=2$ since $B_2-P_2=P_B^1-P_1$, which is either empty or has size exactly $|P_B^1|-|B|+|A_2-P_{A_2}|+1\leq |A_2-P_{A_2}|$. The second inequality holds for $i=1$  by the definition of $B_1$ and by~\eqref{P1}. It holds for $i=2$ since $B_2=B-P_1$ and $|P_1|\leq |B|-|A_2-P_{A_2}|-1$ by the choice of $P_1$.
We may therefore apply Lemma~\ref{lem:pathsBIP} separately to the pairs $A_i,B_i$ with paths $P_{A_i}$ and $P_i$ (recall that $P_i \neq \emptyset$). This gives the desired partition.
\end{proof}

The proof of Theorem~\ref{thm:2colours} splits into two cases, depending on whether or not there is a colour which sees every vertex. When there is a colour which sees every vertex, Lemma~\ref{lem:1coverall} gives the required result. When there is no such colour, it is easy to argue that the graph contains exactly $3$ colours and looks like the configuration in Figure \ref{fig:2colours}. The proof that configurations of this type may be coloured by two monochromatic cycles of different colours forms the core of our proof. While the proof of Lemma~\ref{lem:1coverall} relies crucially on the result of Bessy and Thomass\'e, we will only use the much simpler result of Gy\'arf\'as \cite{GG67, Gy83} that every $2$-coloured graph may be partitioned into two monochromatic paths of different colours to handle the remaining configurations.

\begin{figure}
  \centering
  \begin{tikzpicture}
    [scale=.6,
    V/.style = {draw, fill=white,circle, inner sep =
      0em, outer sep = .1em, minimum size=5em},
    b/.style={ball color = gray,circle,inner sep = 0,minimum
      size=2mm},
    e/.style={line width=.5mm,line cap = rect}
    ]
    \begin{scope}[shift={(10,0)}]
      
      \node (x) at (90:2){};
      \node (y) at (210:2){};
      \node (z) at (330:2){};
      
      \draw (x)--(y);
      \draw (y)--(z);
      \draw (z)--(x);
      
      \draw[circle,fill=white] (90:2) circle (2.4em) {};
      \draw[circle,fill=white] (210:2) circle (1.8em) {};
      \draw[circle,fill=white] (330:2) circle (2.1em) {};

      \node at (90:2) {$V_{1 2}$};
      \node at (210:2) {$V_{2 3}$};
      \node at (330:2) {$V_{1 3}$};
      
      \node at (150:1.5) {$2$};
      \node at (270:1.5) {$3$};
      \node at (30:1.5) {$1$};

    \end{scope}
    \end{tikzpicture}
    \caption{The special configuration in Theorem~\ref{thm:2colours}. All edges with an endpoint in $V_{ij}$ must have colour $i$ or colour $j$. In particular, all edges between $V_{i j}$ and $V_{i \ell}$ have colour $i$.} 
    \label{fig:2colours}
\end{figure}
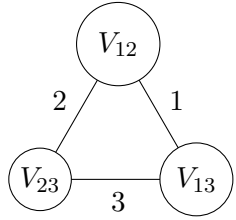

\vspace{3mm}
\noindent
{\it Proof of Theorem~\ref{thm:2colours}.}
Suppose that the edges of $K_n$ have been $2$-locally coloured. Let $S$ be the vertex set of a largest monochromatic connected subgraph, say in colour $1$. If colour $1$ sees all the vertices, then we are done by Lemma~\ref{lem:1coverall}. So $V_{23}:=K_n-S$ is non-empty.

 Let $x\in V_{23}$. Then, by the choice of $S$, we know that all edges between $x$ and $S$ receive at most two colours, both different from~$1$. Suppose these colours are $2$ and $3$. For $i = 2, 3$, let $V_{1i} = S \cap N_i(x)$. Note that these sets partition $S$ and are non-empty (otherwise $\{x\} \cup S$ is a larger monochromatic connected component). Moreover, $V_{1i}$ must contain only the colours $1$ and $i$ and the bipartite graph between $V_{12}$ and $V_{13}$ must be monochromatic in colour $1$. Since every vertex in $V_{1i}$ sees edges with colours $1$ and $i$, we also see that the bipartite graph between $V_{1i}$ and $V_{23}$ must be monochromatic in colour $i$. Finally, this implies that the set $V_{23}$ only contains edges of colours $2$ and $3$. This situation is illustrated in Figure~\ref{fig:2colours}.
 
If $|V_{23}| = 1$, we first ignore the single vertex $v$ in $V_{23}$ and apply Lemma~\ref{lem:1coverall} to find a partition of $V_{12} \cup V_{13}$ into two monochromatic cycles, one of colour $1$ and one of colour $i \in \{2, 3\}$. Since the second cycle will be contained entirely in $V_{1i}$ and all edges between $v$ and $V_{1i}$ are coloured $i$, it is straightforward to extend this cycle to include $v$. We may therefore assume that $|V_{23}| \geq 2$ and, similarly, that $|V_{12}| \geq 2$ and $|V_{13}| \geq 2$.

By Gy\'arf\'as' observation, we know that each $V_{ij}$ has a partition into two paths $R_{ij}^i$ and $R_{ij}^j$, the first of colour $i$ and order $r^i_{ij}$ and the second of colour $j$ and order $r^j_{ij}$ (the {\it order} of a path counts the number of vertices on it). 
We may assume that our graph does not admit a partition into two monochromatic cycles of different colours, as otherwise we would be done. If we apply Lemma~\ref{lem:pathsTRI} with $B = V_{j \ell}$, $A_j = V_{i j}$ and $A_\ell = V_{i \ell}$, then we see that for any choice of paths as above (six paths in total), we have $|V_{i j} - R_{ij}^j| + |V_{i \ell} - R_{i \ell}^\ell| + 1 \geq |V_{j \ell}|$ provided $r^j_{ij},r^\ell_{i\ell}>0$. That is, 
 \begin{equation}\label{length}
 \text{ if $r^j_{ij},r^\ell_{i\ell}>0$, then $r^i_{ij}+r^i_{i\ell} +1\geq  |V_{j\ell}|$}
 \end{equation}
  for any distinct $i,j,\ell\in\{1,2,3\}$. 
  
We now fix paths $P^i_{ij}$ and $P^j_{ij}$ partitioning $V_{ij}$ for each distinct $i,j\in\{1,2,3\}$ and set $p^i_{ij}:=|P^i_{ij}|$ and $p^j_{ij}:=|P^j_{ij}|$. Since $|V_{ij}| \geq 2$ for all $i,j\in\{1,2,3\}$, we may assume that each of our paths has at least one vertex. Slightly abusing notation by setting $p^i_{ij}:=p^i_{ji}$ whenever necessary, we clearly have
  \begin{equation}\label{sumpi}
 |V_{12}|+ |V_{13}|+ |V_{23}|= \sum_{i,j\in \{1,2,3\}, i< j}(p^i_{ij}+p^j_{ij})= \sum_{i,j,\ell\in \{1,2,3\}, \ell\neq i<j\neq \ell}(p^\ell_{i\ell}+p^\ell_{j\ell}).
  \end{equation}
  Hence, for at least two of the three sets $V_{ij}$, we have 
 $|V_{ij}|\geq p^\ell_{i\ell}+p^\ell_{j\ell}$, where $i\neq \ell\neq j$ (if this inequality was false for two of the sets, then the third set would violate~\eqref{length}). Without loss of generality, we will assume that
 \begin{equation}\label{V12V13}
  |V_{12}|\geq p^3_{13}+p^3_{23}\quad \text{ and }\quad  |V_{13}|\geq p^2_{12}+p^2_{23}.
  \end{equation}
 
 Now consider the edges between the endpoints of the paths $P^i_{ij}$ and $P^j_{ij}$ for distinct $i,j \in\{1,2,3\}$. These edges have colours in $\{i,j\}$. Therefore, there is always one of the two paths which can be made shorter by $1$ (possibly becoming empty), while augmenting the other path by $1$. Even more is true: if  one of the two paths, say $P^i_{ij}$, cannot be shortened\footnote{A bit incorrectly, we say a path $P^i_{ij}$ can/cannot be shortened by $x$ if the other path $P^j_{ij}$ can/cannot be augmented by $x$ using the endpoints of $P^i_{ij}$.} by~$1$, then  the other one, $P^j_{ij}$, either can be shortened by at least $2$ or is a one-vertex path. In the latter case, observe that $P^i_{ij}$ can be extended to a spanning cycle of $V_{ij}$.
  
Apply this reasoning to the paths $P^i_{23}$, $i=2,3$. By~\eqref{length}, neither of these paths may be shortened by $p^i_{1i}+p^i_{23} +2 -|V_{1(5-i)}|$ (note that this is either $1$ or $2$). Otherwise, $V_{23}$ would have a partition into two paths, with the path of colour $i$ having order $q_{23}^i = |V_{1(5-i)}| - 2 - p_{1i}^i$. Since  
 \[p_{1i}^i + q_{23}^i = |V_{1(5-i)}| - 2,\]
 this would contradict~\eqref{length}. If now each of the paths $P^i_{23}$, $i=2,3$, may be shortened by one, this implies that $p^i_{1i}+p^i_{23} +2 -|V_{1(5-i)}| \geq 2$ and so 
 $|V_{12}| = p^3_{13}+p^3_{23}$ and $|V_{13}|= p^2_{12}+p^2_{23}$, by~\eqref{V12V13}. Next, assume that one of the paths cannot be shortened by $1$. Hence, the other path, say $P_{23}^2$, can either be shortened by $2$ or consists of only one vertex. The first case cannot hold since then $p^2_{12}+p^2_{23} +2 -|V_{13}| \geq 3$ and so $|V_{13}| \leq p^2_{12} + p^2_{23} - 1$, contradicting~\eqref{V12V13}. In the second case, since we can shorten $P^2_{23}$ by one, we can again argue that $|V_{13}|= p^2_{12}+p^2_{23}$. Therefore, one of the following holds, after possibly swapping colours $2$ and $3$ for (b),
\begin{enumerate}[(a)]
\item  each of the paths $P^i_{23}$, $i=2,3$, may be shortened by one and, thus, $|V_{12}|= p^3_{13}+p^3_{23}$ and $|V_{13}|= p^2_{12}+p^2_{23}$,
\item $p^2_{23}=1$, $|V_{13}|= p^2_{12}+p^2_{23}$ and $V_{23}$ has a spanning cycle in colour $3$.
 \end{enumerate}
 
 In case~(a), note that by~\eqref{sumpi}, we also have $|V_{23}|= p^1_{12}+p^1_{13}$. In each of $V_{1i}$, $i=2,3$, one of the two paths $P^1_{1i}$, $P^i_{1i}$ can be shortened by $1$. However, by~\eqref{length}, not both $P^1_{12}$ and $P^1_{13}$ can be shortened by one. Thus, one of the $P^i_{1i}$, $i=2,3$, can be shortened by $1$. Since, by~(a), $P^i_{23}$ can also be shortened by one, this contradicts~\eqref{length}.
 
 So assume (b) holds. Then $|V_{13}|= p^2_{12}+p^2_{23}= p^2_{12}+1$. We now wish to apply Lemma~\ref{lem:pathsBIP} with $A:=V_{12}$, $B:=V_{13}$, $P_A:=P^1_{12}$ and $P_B$ any one-vertex path in $V_{13}$. This is possible since both paths are nonempty and 
 $$|A - P_A| = |V_{12}| - p_{12}^1 = p_{12}^2 = |V_{13}| - 1 = |B| - 1.$$
 Hence, there is a cycle in colour $1$ that covers all of $V_{12}\cup V_{13}$. Together with the spanning cycle of $V_{23}$ in colour~$3$, this gives the desired partition.
 \qed

\section{Concluding remarks} \label{sec:conclude}

\noindent
{\bf Constructing counterexamples.} The topic of this paper was originally motivated by an attempt to construct further counterexamples to the original conjecture. Concretely, suppose that one has an $(s-1)$-locally coloured complete graph containing $s$ colours in total and that at least $s$ monochromatic cycles are necessary to cover all vertices. It will also be useful to assume that this property is somewhat robust. For our purposes, it will be sufficient to know that $s$ monochromatic cycles are still needed to cover the graph whenever any vertex is deleted.\footnote{It is worth noting that Pokrovskiy's example \cite{P14} of a $3$-coloured complete graph which requires four monochromatic cycles is not robust in this sense, since one may cover all but one vertex with three monochromatic cycles.}

We now add an additional vertex $u$ to the graph. For each vertex $v$, we give $uv$ a colour which did not appear at $v$ in the original colouring. If we try to cover this new graph with monochromatic cycles, we see that $u$ must appear either on its own or as part of a single edge. If it occurs on its own, it is clear that $s+1$ monochromatic cycles are needed to cover the graph. If it appears as part of a single edge $uv$, we delete $u$ and $v$. But then the robustness property tells us that the graph that remains requires at least $s$ monochromatic cycles to cover all vertices, so we needed $s+1$ cycles in total.

Unfortunately, since we found no $(s-1)$-locally coloured complete graphs with the required properties, we could not use this technique to produce further counterexamples to the original conjecture. However, it may yet be a fruitful direction to consider. 

\vspace{3mm}
\noindent
{\bf The structure of $r$-local colourings.} When proving Theorem~\ref{thm:2colours}, we saw that we may split our deduction into two cases, depending on whether or not there was a colour which was seen by all vertices. The case where one colour sees all vertices followed easily as a corollary of the Bessy--Thomass\'e result, while the case where not all vertices are seen by one colour devolved into a special case which we had to study in depth.

As noted in \cite{GLST87}, a similar methodology can be applied to $r$-local colourings. Suppose that we have an $r$-locally coloured complete graph using $s$ colours in total, which we may assume to be $1, 2, \dots, s$. We form a hypergraph on the vertex set $\{1, 2, \dots, s\}$ by letting $\{c_1, c_2, \dots, c_i\}$ with $i \leq r$ be an edge if and only if there exists a vertex $v$ which sees precisely the colours $c_1, c_2, \dots, c_i$. Since any two vertices in our complete graph have an edge between them, any two edges in this hypergraph must intersect. A result of Erd\H{o}s and Lov\'asz \cite{EL75} now implies that either there are $r-1$ vertices such that every edge in the hypergraph contains at least one of these vertices or the hypergraph has at most $r^r$ edges. Translated back to the original setting, either there are $r-1$ colours such that every vertex sees at least one of these colours or every vertex sees one of at most $r^r$ different colour combinations.

For $r = 2$, this result again reduces to saying that either there is a colour which is seen by every vertex or we have a configuration of the type illustrated in Figure~\ref{fig:2colours}. For $r = 3$, it tells us that there are two colours, say $1$ and $2$, such that either every vertex sees at least one of $1$ and $2$ or every vertex sees one of at most $27$ different colour combinations (actually, as noted in \cite{GLST87}, this may be reduced to $10$). While this certainly gives us substantial extra information, a detailed analysis of monochromatic covers in $3$-local colourings is likely to be unwieldy. For example, one would have to analyse colourings corresponding to the Fano plane illustrated in Figure~\ref{fig:fano}. 

\begin{figure}
  \centering
  \begin{tikzpicture}
    [scale=.6,
    V/.style = {draw, fill=white,circle, inner sep =
      0em, outer sep = .1em, minimum size=5em},
    b/.style={ball color = gray,circle,inner sep = 0,minimum
      size=2mm},
    e/.style={line width=.5mm,line cap = rect}
    ]
    \begin{scope}[shift={(10,0)}]

      \node (a) at (90:3.5) {};
      \node (b) at (141.428:3.5) {};
      \node (c) at (192.857:3.5) {};
      \node (d) at (244.285:3.5) {};
      \node (e) at (295.714:3.5) {};
      \node (f) at (347.142:3.5) {};
      \node (g) at (38.571:3.5) {};

     \draw (a)--(b);
     \draw (a)--(c);
     \draw (a)--(d);
     \draw (a)--(e);
     \draw (a)--(f);
     \draw (a)--(g);
     \draw (c)--(b);
     \draw (d)--(b);
     \draw (e)--(b);
     \draw (f)--(b);
     \draw (g)--(b);
     \draw (c)--(d);
     \draw (c)--(e);
     \draw (c)--(f);
     \draw (c)--(g);
     \draw (d)--(e);
     \draw (d)--(f);
     \draw (d)--(g);
     \draw (e)--(f);
     \draw (e)--(g);
     \draw (f)--(g);

      \draw[circle,fill=white] (90:3.5) circle (2.5em) {};
      \draw[circle,fill=white] (141.428:3.5) circle (2.1em) {};
      \draw[circle,fill=white] (192.857:3.5) circle (1.9em) {};
      \draw[circle,fill=white] (244.285:3.5) circle (2.4em) {};
      \draw[circle,fill=white] (295.714:3.5) circle (1.8em) {};
      \draw[circle,fill=white] (347.142:3.5) circle (2.3em) {};
      \draw[circle,fill=white] (38.571:3.5) circle (2em) {};

      \node at (90:3.5) {$V_{1 3 7}$};
      \node at (141.428:3.5) {$V_{2 6 7}$};
      \node at (192.857:3.5) {$V_{1 5 6}$};
      \node at (244.285:3.5) {$V_{4 5 7}$};
      \node at (295.714:3.5) {$V_{3 4 6}$};
      \node at (347.142:3.5) {$V_{2 3 5}$};
      \node at (38.571:3.5) {$V_{1 2 4}$};
      
    \end{scope}
  \end{tikzpicture}
  \caption{A $3$-local colouring corresponding to the Fano plane. All edges with an endpoint in $V_{i j k}$ must have colour $i$, $j$ or $k$. In particular, all edges between $V_{i j k}$ and $V_{i \ell m}$ have colour $i$.}
  \label{fig:fano}
\end{figure}
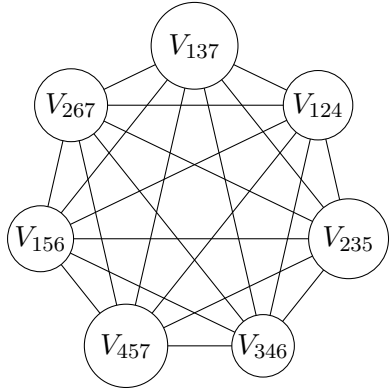

It might also be interesting to restrict the total number of colours. To give an example, a result of Pokrovskiy~\cite{P14} states that the vertex set of any $3$-coloured complete graph may be partitioned into at most three monochromatic paths. Perhaps one can prove that this remains true in any $3$-locally coloured complete graph containing at most $4$ colours.

\vspace{3mm}
\noindent
{\bf Mean colourings.}
An edge colouring of a graph is said to be an $r$-mean colouring if the average number of colours incident to any vertex is at most $r$. The $r$-mean Ramsey number of a graph $H$, denoted $R_{r-mean}(H)$, is the smallest $n$ such that in any $r$-mean colouring of $K_n$ there is guaranteed to be a monochromatic copy of $H$. This concept was introduced by Caro \cite{C92}, who also proved that $R_{r-mean}(H)$ exists for all $H$. The proof is quite simple: we find a large subset on which the colouring is an $(r+1)$-local colouring and then apply the existence of local Ramsey numbers.

Given the ease with which the concept of local Ramsey numbers generalises to mean Ramsey numbers, it is worth asking whether the vertex set of any $r$-mean coloured $K_n$ can be partitioned into a finite number of disjoint monochromatic cycles. We have been unable to resolve this question in general. However, for $r = 2$, we can again show that two monochromatic cycles of different colours are sufficient to cover all vertices of $K_n$. We now sketch the proof.

To begin, note that the set of vertices $V_1$ seeing exactly one colour is at least as large as the set of vertices $V_3$ seeing three or more colours. In particular, if $V_1$ is empty, then so is $V_3$ and Theorem~\ref{thm:2colours} applies. We may therefore assume that $V_1 \neq \emptyset$. Note that all vertices in $V_1$ must see the same colour, which we assume to be colour $1$. It is now straightforward to choose a cycle in colour $1$ which covers all vertices in $V_1 \cup V_3$. However, we will instead choose a cycle $C_1$ which covers all but one vertex $v$ of $V_3$ and which uses an internal edge $e$ in $V_1$ (unless $|V_1| = 1$, in which case the cycle is just a singleton). Consider now the set of vertices $V_2$ seeing exactly two colours. Since $V_1 \neq \emptyset$, every vertex in $V_2$ sees colour $1$. 
We may therefore apply Lemma~\ref{lem:onemore} to conclude that $V_2 \cup \{v\}$ may be covered by two monochromatic cycles of different colours, say $C'_1$ and $C_2$, with $C'_1$ of colour $1$. Using the edge $e$ (or the single vertex if $|C_1| = 1$), we may now combine $C_1$ and $C'_1$ into a cycle, completing the proof.

\vspace{3mm}
\noindent
{\bf Improving the bounds.}
We have proved that the vertex set of any $r$-locally coloured complete graph may be partitioned into $O(r^2 \log r)$ monochromatic cycles. It would be interesting to know whether these bounds can be substantially improved. While we cannot hope that $r$ cycles are always enough, it seems plausible that $O(r)$ is. It would already be interesting to bring the bounds in line with the $O(r \log r)$ bound of Gy\'arf\'as, Ruszink\'o, S\'ark\"ozy and Szemer\'edi \cite{GRSS06}.


\begin{thebibliography}{}

\bibitem{A08}
{P. Allen,} {Covering two-edge-coloured complete graphs with two disjoint monochromatic cycles,}
{\it Combin. Probab. Comput.} {\bf 17} (2008), 471--486. 

\bibitem{Ay79}
{J. Ayel,} {Sur l'existence de deux cycles suppl\'ementaires unicolores, disjoints et de couleurs diff\'erentes dans un graph complet bicolore,}
{PhD thesis, Universit\'e de Grenoble}, 1979.

\bibitem{BT10} 
{S. Bessy and S. Thomass\'e,} {Partitioning a graph into a cycle and an anticycle: a proof of Lehel's conjecture,} {\it J. Combin. Theory Ser. B} {\bf 100} (2010), 176--180.

\bibitem{BKS00}
{B. Bollob\'as, A. Kostochka and R. H. Schelp,} {Local and mean Ramsey numbers for trees,} {\it J. Combin. Theory Ser. B} {\bf 79} (2000), 100--103.

\bibitem{C92}
{Y. Caro,} {On several variations of the Tur\'an and Ramsey numbers,} {\it J. Graph Theory} {\bf 16} (1992), 257--266.

\bibitem{EG59}
{P. Erd\H{o}s and T. Gallai,} {On maximal paths and circuits of graphs,} {\it Acta. Math. Acad. Sci. Hungar.} {\bf 10} (1959), 337--356. 

\bibitem{EGP91}
{P. Erd\H{o}s, A. Gy\'arf\'as and L. Pyber,} {Vertex coverings by monochromatic cycles and trees,}
{\it J. Combin. Theory Ser. B} {\bf 51} (1991), 90--95. 

\bibitem{EL75}
{P. Erd\H{o}s and L. Lov\'asz,} {Problems and results on $3$-chromatic hypergraphs and some related questions,} in
{Infinite and finite sets (Colloq., Keszthely, 1973), Vol. II,} 609--627, Colloq. Math. Soc. J\'anos Bolyai, Vol. 10, North-Holland, Amsterdam, 1975.

\bibitem{GG67}
{L. Gerencs\'er and A. Gy\'arf\'as,} {On Ramsey-type problems,} {\it Annales Univ. E\"otv\"os Section Math.} {\bf 10} (1967), 167--170.

\bibitem{Gy83}
{A. Gy\'arf\'as,} {Vertex coverings by monochromatic paths and cycles,} {\it J. Graph Theory} {\bf 7} (1983), 131--135.

\bibitem{GLST87}
{A. Gy\'arf\'as, J. Lehel, R. H. Schelp and Zs. Tuza,} {Ramsey numbers for local colorings,} 
{\it Graphs Combin.} {\bf 3} (1987), 267--277.

\bibitem{GRSS06}
{A. Gy\'arf\'as, M. Ruszink\'o, G. S\'ark\"ozy and E. Szemer\'edi,} {An improved bound for the monochromatic cycle partition number,}
{\it J. Combin. Theory Ser. B} {\bf 96} (2006), 855--873.

\bibitem{GRSS11}
{A. Gy\'arf\'as, M. Ruszink\'o, G. S\'ark\"ozy and E. Szemer\'edi,} {Partitioning $3$-colored complete graphs into three monochromatic cycles,}
{\it Electron. J. Combin.} {\bf 18} (2011), Paper 53, 16 pp.

\bibitem{L07}
{L. Lov\'asz,} {\bf Combinatorial problems and exercises,} 2nd ed., AMS Chelsea Publishing, 2007.

\bibitem{LRS98}
{T. \L uczak, V. R\"odl and E. Szemer\'edi,} {Partitioning two-coloured complete graphs into two monochromatic cycles,}
{\it Combin. Probab. Comput.} {\bf 7} (1998), 423--436. 

\bibitem{P14}
{A. Pokrovskiy,} {Partitioning edge-coloured complete graphs into monochromatic cycles and paths,} {\it J. Combin. Theory Ser. B}  {\bf 106} (2014), 70--97.


\bibitem{TT87}
{M. Truszczynski and Zs. Tuza,} {Linear upper bounds for local Ramsey numbers,}
{\it Graphs Combin.} {\bf 3} (1987), 67--73.

\end{thebibliography}
\end{document}